\documentclass[a4paper,12pt]{article}
\usepackage[top=2.5cm,bottom=2.5cm,left=2.5cm,right=2.5cm]{geometry}
\usepackage{cite, amsmath, amssymb, cases}
\usepackage{amssymb}
\usepackage{graphicx}
\newenvironment{proof}{{\noindent\textbf{\textit {Proof.}}}}{\hfill $\blacksquare$\par}
\usepackage{latexsym}
\usepackage{graphicx,booktabs,multirow}
\usepackage{tikz}
\usetikzlibrary{decorations.pathreplacing}
\usetikzlibrary{intersections}
\usepackage{enumerate}
\usepackage{graphicx,booktabs,multirow}
\usepackage{appendix}
\newtheorem{theorem}{Theorem}[section]
\newtheorem{proposition}[theorem]{\rm\bfseries Proposition}
\newtheorem{lemma}[theorem]{Lemma}
\newtheorem{corollary}[theorem]{\rm\bfseries Corollary}
\newtheorem{remark}[theorem]{Remark}

\usepackage[section]{placeins}
\def\NAT@def@citea{\def\@citea{\NAT@separator}}% Suppress spaces between citations using natbib.sty
\allowdisplaybreaks

\begin{document}
\vspace*{10mm}

\noindent
{\Large \bf Spectral conditions for spanning $k$-trees or $k$-ended-trees of $t$-connected graphs}

\vspace*{7mm}

\noindent
{\large \bf  Jifu Lin$^1$, Zenan Du$^{2,*}$, Xinghui Zhao$^1$, Lihua You$^1$}
\noindent

\vspace{7mm}

\noindent
$^1$ School of Mathematical Sciences, South China Normal University,  Guangzhou, 510631, P. R. China,
e-mail: {\tt 2023021893@m.scnu.edu.cn (J.F. Lin)},\\ {\tt 2022021990@m.scnu.edu.cn (X.H. Zhao)}, {\tt ylhua@scnu.edu.cn (L.H. You)}.\\[2mm]
$^2$ School of Mathematics and Statistics, Shanxi University, Shanxi, 030006, P. R. China, e-mail: {\tt duzn@sxu.edu.cn (Z.N. Du)}.\\[2mm]
$^*$ Corresponding author
\noindent

%\footnotesize $^1${\it School of Mathematical Sciences, South China Normal University, Guangzhou, 510631, P. R. China}\\
%\footnotesize $^2${\it Department of Mathematics Teaching, Guangzhou Civil Aviation College, Guangzhou, 510403, P. R. China}\\
%\noindent
% $^2${\it Department of Mathematics Teaching, Guangzhou Civil Aviation College, Guangzhou, 510403, P. R. China\/} \\
\vspace{7mm}

\noindent
{\bf Abstract} \
\noindent
Let $G$ be a connected graph of order $n$. A spanning $k$-tree of $G$ is a spanning tree with the maximum degree at most $k$, and a spanning $k$-ended-tree of $G$ is a spanning tree at most $k$ leaves, where $k\geq2$ is an integer. This paper establishes some spectral conditions for the existence of spanning $k$-trees or spanning $k$-ended-trees in $t$-connected graphs, which generalize the results of Fan et al. (2022) and Zhou (2010), and improve the results of Fiedler et al. (2010), Ao et al. (2023) and Ao et al. (2025).
 \\[2mm]
%\vspace{5mm}

\baselineskip=0.30in

\section{Introduction}

\hspace{1.5em}Throughout this paper, we only consider simple and undirected graphs. 

Let $G=(V(G),E(G))$ be a graph with vertex set $V(G)=\{v_1,v_2,\cdots,v_n\}$ and edge set $E(G)$, where $|V(G)|=n$ and $|E(G)|=e(G)$ are the \emph{order} and \emph{size} of $G$, respectively. For $v\in V(G)$, the degree of $v$ in $G$, denoted by $d_G(v)$(or $d(v)$ for short), is the number of vertices adjacent to $v$ in $G$. The minimum degree of $G$ is denoted by $\delta(G)$. For a subset $S \subseteq V(G)$, we denote by $G[S]$ the subgraph of $G$ induced by $S$, and by $G-S$ the subgraph obtained from $G$ by removing the vertices in $S$ and their incident edges. We use $K_n$ and $R(n,t)$ to denote the complete graph and a $t$-regular graph of order $n$, respectively. The \emph{complement} of a graph $G$ is denoted by $\overline{G}$.

Let $G_1$ and $G_2$ be vertex disjoint graphs. The \emph{union} $G_1\cup G_2$ is the graph with vertex set $V(G_1)\cup V(G_2)$ and edge set $E(G_1)\cup E(G_2)$. For any positive integer $t$, let $tG$ be the disjoint union of $t$ copies of $G$. The \emph{join} $G_1 \vee G_2$ is obtained from $G_1 \cup G_2$ by joining each vertex of $G_1$ to each vertex of $G_2$. Other undefined notations can be found in \cite{JAB}.

Let $A(G)=[a_{ij}]$ be the $n\times n$ \emph{adjacency matrix} of $G$ for which $a_{ij}=1$ if $v_i$ and $v_j$ are adjacent in $G$ and $a_{ij}=0$ otherwise. Let $D(G)$ be the diagonal matrix of vertex degrees in $G$. The matrix $Q(G)=D(G)+A(G)$ is called the \emph{signless Laplacian matrix} of $G$. The largest eigenvalue of $A(G)$ (resp. $Q(G)$), denoted by $\rho(G)$ (resp. $q(G)$), is called the adjacency spectral radius (resp. signless Laplacian spectral radius) of $G$.

A \emph{spanning tree} $T$ of a connected graph $G$ is a connected spanning subgraph of $G$ with $V(T)=V(G)$ and $|E(T)|=|V(T)|-1$. For an integer $k\geq 2$, a $k$-$tree$ is a tree with the maximum degree at most $k$, and a $k$-ended-tree is a tree with at most $k$ leaves. In particular, a spanning $2$-tree or spanning $2$-ended-tree is a Hamilton path of $G$. Hence a spanning $k$-tree or spanning $k$-ended-tree of $G$ is a natural generalization of a Hamilton path. A connected graph $G$ is called $t$-\emph{connected} if it has more than $t$ vertices and remains connected whenever fewer than $t$ vertices are removed.

In graph theory, determining the existence of spanning trees with specific structural properties is a fundamental problem. A famous example is the Hamilton path problem, which aims to verify whether a graph contains a spanning $2$-tree or $2$-ended-tree. In the past decades, the study on the existence of spanning $k$-trees or $k$-ended-trees attracted much attention. For example, Win \cite{S1} made a connection between the existence of spanning $k$-trees in a graph and its toughness, and provided a Chv\'{a}tal-Erd\H{o}s type condition to ensure that a $t$-connected graph contains a spanning $k$-ended tree \cite{S2}; Ellingham and Zha \cite{ME} presented a short proof to Win's theorem; Gu and Liu \cite{XA} characterized a connected graph with a spanning $k$-tree by applying Laplacian eigenvalues; Zhou, Zhang and Liu \cite{SZ} investigated the relations between the spanning $k$-tree and distance signless Laplacian spectral radius; Ozeki and Yamashita \cite{KO} proved that it is an $\mathcal{NP}$-complete problem to decide whether a given connected graph admits a spanning $k$-tree; Ao, Liu and Yuan \cite{GYA1} proved some spectral conditions to ensure the existence of a spanning $k$-ended-tree in connected graphs; Zheng, Huang and Wang \cite{JZ} obtained a adjacent spectral condition for spanning $k$-ended-trees in $t$-connected graphs.

In this paper, we study the existence of spanning $k$-trees and spanning $k$-ended-trees in $t$-connected graphs, obtain some spectral conditions as shown in Sections \ref{sec3} and Section \ref{sec4}, respectively, which generalize the results of Fan et al. \cite{DT} and Zhou \cite{BZ}, and improve the results of Fiedler et al. \cite{MF} and Ao et al. \cite{GYA,GYA1}. 

\section{Preliminaries}\label{sec-pre}

\hspace{1.5em}In this section, we introduce some preliminary lemmas, which are useful in the following proofs.

\begin{lemma}{\rm(\!\!\cite{ylh})}\label{rq1}
	Let $G$ be a graph and $F$ be a spanning subgraph of $G$. Then $\rho(F)\leq \rho(G)$ and $q(F)\leq q(G)$. In particular, if $G$ is connected and $F$ is a proper subgraph of $G$, then $\rho(F)<\rho(G)$ and $q(F)<q(G)$.
\end{lemma}

\begin{lemma}{\rm(\!\!\cite{YH, VNS})}\label{r1}
	Let $G$ be a graph of order $n$ with the minimum degree $\delta(G)\geq t$. Then $$\rho(G)\leq\frac{t-1+\sqrt{(t+1)^2+4(2e(G)-nt)}}{2}.$$
\end{lemma}

\begin{lemma}{\rm(\!\!\cite{MH})}\label{r2}
	Let $G$ be a graph of order $n$. Then $\sum\limits_{v\in V(G)}d(v)^2\leq n\rho(G)^2$.
\end{lemma}

\begin{lemma}{\rm(\!\!\cite{BL})}\label{r3}
	Let $G$ be a graph with non-empty edge set. Then $$\rho(G)\geq\min\{\sqrt{d(u)d(v)}: uv\in E(G)\}.$$ Moreover, if $G$ is connected then equality holds if and only if $G$ is regular or semi-regular. 
\end{lemma}

\begin{lemma}{\rm(\!\!\cite{KD,LF})}\label{q1}
	Let $G$ be a connected graph of order $n$ and size $e(G)$. Then $$q(G)\leq\frac{2e(G)}{n-1}+n-2.$$
\end{lemma}

\begin{lemma}{\rm(\!\!\cite{BZ})}\label{q2}
	Let $G$ be a graph of size $e(G)$. Then $q(G)\geq \frac {z(G)}{e(G)},$ where $z(G)=\sum\limits_{u\in V(G)}d(u)^2$. Moreover, if $G$ is connected then equality holds if and only if $G$ is regular or semi-regular. 
\end{lemma}

The $l$-$closure$ of a graph $G$, denoted by $C_l(G)$, is the graph obtained from $G$ recursively joining pairs of nonadjacent vertices whose degree sum is at least $l$ until no such pair remains.

\begin{lemma}{\rm(\!\!\cite{GYA})}\label{edge}
	Let $G$ be an $t$-connected graph of order $n\geq (7k-2)t+4$, where $t$ and $k$ are integers with $t\geq1$ and $k\geq2$. If $$e(G)\geq\binom{n-(k-1)t-2}{2}+(k-1)t^2+(k+1)t+3,$$ then $G$ has a spanning $k$-tree unless $C_{n-(k-2)t-1}(G)\cong K_t\vee(K_{n-kt-1}\cup (kt-t+1)K_1)$.
\end{lemma}

\begin{lemma}{\rm(\!\!\cite{JZ})}\label{edge2}
	Let $G$ be an $t$-connected graph of order $n\geq \max\{6k+6t-1,k^2+kt+t+1\}$, where $t$ and $k$ are integers with $t\geq1$ and $k\geq2$. If $$e(G)\geq\binom{n-k-t}{2}+(k+t-1)^2+k+t,$$ then $G$ has a spanning $k$-ended-tree unless $C_{n-1}(G)\cong K_t\vee(K_{n-k-2t+1}\cup (k+t-1)K_1)$.
\end{lemma} 

\begin{lemma}{\rm(\!\!\cite{JZ})}\label{k-ended tree}
	Let $n$, $t\geq 1$, $k\geq 2$ be positive integers and $G=K_t\vee(K_{n-k-2t+1}\cup (k+t-1)K_1)$. Then $G$ does not have spanning $k$-ended-tree.
\end{lemma}

\begin{remark}
	We notice that Lemma \ref{k-ended tree} should hold only if $n\geq k+2t$ since $K_t\vee (k+t-1)K_1$ has a spanning $k$-ended-tree.
\end{remark}

\section{Spanning $k$-trees}\label{sec3}
 \hspace{1.5em} In this section, we establish some spectral conditions for the existence of spanning $k$-trees in $t$-connected graphs, which generalize the results of Fan et al. \cite{DT} and Zhou \cite{BZ}, and improve the results of Fiedler et al. \cite{MF} and Ao et al. \cite{GYA}. Now, we recall some relevant conclusions.
 
 Kano and Kishimoto \cite{MK} proved the following closure theorem to assure that a $t$-connected graph has a spanning $k$-tree.

\begin{theorem}{\rm(\!\!\cite{MK})}\label{closure}
	Let $t\geq1$ and $k\geq2$ be integers, $G$ be a $t$-connected graph of order $n$. Then $G$ has a spanning $k$-tree if and only if the $(n-kt+2t-1)$-closure $C_{n-kt+2t-1}(G)$ of $G$ has a spanning $k$-tree. 
\end{theorem}

Using the toughness-type condition, Fan, Goryainov, Huang and Lin \cite{DT} posed the following spectral radius conditions for the existence of a spanning $k$-tree in a connected graph.

\begin{theorem}{\rm(\!\!\cite{DT})}\label{t0}
	Let $k\geq 3$, $G$ be a connected graph of order $n \geq 2k+16$. If $\rho(G)\geq\rho(K_1\vee(K_{n-k-1}\cup kK_1))$ or $q(G)\geq q(K_1\vee(K_{n-k-1}\cup kK_1)),$  then $G$ contains a spanning $k$-tree unless $G\cong K_1\vee(K_{n-k-1}\cup kK_1)$.
\end{theorem}

Ao, Liu, Yuan, Ng and Cheng \cite{GYA} proved a spectral condition to guarantee the existence of a spanning $k$-tree in a $t$-connected graph.

\begin{theorem}{\rm(\!\!\cite{GYA})}\label{t1}
	Let $G$ be a $t$-connected graph of order $n\geq \max\{(7k-2)t+4,(k-1)t^2+\frac12(3k+1)t+\frac92\}$, where $t\geq1$ and $k\geq2$ are integers. If $\rho(G)\geq\rho(K_t\vee(K_{n-kt-1}\cup (kt-t+1)K_1))$, then $G$ has a spanning $k$-tree unless $G\cong K_t\vee (K_{n-kt-1}\cup (kt-t+1)K_1)$.
\end{theorem}

Fiedler and Nikiforov \cite{MF}, Zhou \cite{BZ} proved spectral conditions of the complement of graphs for the existence of Hamilton path.

\begin{theorem}{\rm(\!\!\cite{MF})}\label{t3}
	Let $G$ be a connected graph of order $n$ with its complement $\overline{G}$. If $\rho(\overline{G})\leq \sqrt{n-1}$, then $G$ contains a Hamilton path.
\end{theorem}

Let $\mathbb{EP}_n$ be the set of graphs on $n$ vertices of three types: (a) a regular graph of degree $\frac{n}{2} - 1$; (b) a graph consisting of two complete components; (c) the join of a regular graph of degree $\frac{n}{2} -1- r$ and a graph on $r$ vertices, where $1 \leq r \leq \frac{n}{2}-1$.

\begin{theorem}{\rm(\!\!\cite{BZ})}\label{t4}
	Let $G$ be a graph of order $n$ with its complement $\overline{G}$. If $q(\overline{G})\leq n$ and $G\notin \mathbb{EP}_n$, then $G$ contains a Hamilton path.
\end{theorem}

Motivated by \cite{GYA,DT,MK,MF,BZ}, we study the sufficient conditions to ensure that a $t$-connected graph $G$ has a spanning $k$-tree in terms of $\rho(G)$ or $\rho(\overline{G})$. Firstly, we derive the following proposition.

\begin{proposition}\label{k-tree}
	Let $n$, $t\geq 1$ and $k\geq 2$ be integers with $n\geq kt+2$, $H^*=K_t\vee(K_{n-kt-1}\cup(kt-t+1)K_1)$. Then the maximum degree of each spanning tree of $H^*$ is at least $k+1$, and thus $H^*$ has no spanning $k$-trees.
\end{proposition}

\begin{proof}
	Let $T$ be a spanning tree of $H^*$, $V_1=V(K_t)$, $V_2=V((kt-t+1)K_1)$ and $V_3=V(K_{n-kt-1})$. Clearly, $V_1$, $V_2$ and $V_3$ are not empty sets. Then $T_1=T[V_1\cup V_2]$ is a forest and $e(T_1)=kt+1-c$, where $c$ is the number of components of $T_1$.
	
	Let $e(V_1,V_3)$ be the number of edges between $V_1$ and $V_3$ in $T$. Then $e(V_1,V_3)\geq c$ since $T$ is connected and $e(T_1)=kt+1-c$, and thus $\sum\limits_{v\in V_1}d_T(v)\geq e(T_1)+e(V_1,V_3)\geq kt+1$. Therefore, there exists $v\in V_1\subseteq V(T)$ such that $d_T(v)\geq k+1$ by the Pigeonhole Principle, and $H^*$ has no spanning $k$-trees.
\end{proof}

\begin{theorem}\label{ta}
	Let $t\geq 1$ and $k\geq2$ be integers, $H^*=K_t\vee(K_{n-kt-1}\cup(kt-t+1)K_1)$, $G$ be a $t$-connected graph of order $n$. 
	
	{\rm (i)} If $n\geq \max\{(7k-2)t+4,\frac 12(k-1)t^2+2kt-\frac t2+5\}$ and $\rho(G)\geq\rho(H^*)$, then $G$ has a spanning $k$-tree unless $G\cong H^*$.
	
	{\rm (ii)} If $n\geq \max\{(7k-2)t+4,\frac 12(k^2-1)t^2+\frac 12(5k-1)t+5\}$ and $q(G)\geq q(H^*)$, then $G$ has a spanning $k$-tree unless $G\cong H^*$.
\end{theorem}

\begin{proof}
	 Suppose to the contrary in (i) and (ii) that $G$ has no spanning $k$-trees. Now we show $G\cong H^*$, where $H^*$ has no spanning $k$-trees by Proposition \ref{k-tree}. 
	 
	 Clearly, $K_{n-(k-1)t-1}\cup(kt-t+1)K_1$ is a proper spanning subgraph of $H^*$. Then by Lemma \ref{rq1}, we have\begin{equation}\label{eq-1}
	 	\rho(H^*)>\rho(K_{n-(k-1)t-1}\cup(kt-t+1)K_1)=n-(k-1)t-2 ,
	 \end{equation}
	 \begin{equation}\label{eq0}
	 q(H^*)>q(K_{n-(k-1)t-1}\cup(kt-t+1)K_1)=2(n-(k-1)t-2).
	 \end{equation} 
	 
	 Now, we show (i)-(ii) in sequence.
	 
	 (i) Since $G$ is $t$-connected, we have $\delta(G)\geq t$. Then by Lemma \ref{r1} and (\ref{eq-1}), we have
	\begin{equation}\label{eq1}
		n-(k-1)t-2<\rho(H^*)\leq \rho(G)\leq \frac{t-1+\sqrt{(t+1)^2+4(2e(G)-nt)}}{2}.\end{equation}
		By (\ref{eq1}) and direct computation, we have $$\begin{aligned}
		e(G)>&\frac{(n-(k-1)t-2)^2-(t-1)(n-(k-1)t-2)-t+nt}{2}\\\geq&\binom{n-(k-1)t-2}{2}+(k-1)t^2+(k+1)t+3,
	\end{aligned}$$ where $n\geq \frac12(k-1)t^2+2kt-\frac t2+5$. 
	
	Let $H=C_{n-(k-2)t-1}(G)$. By $n\geq(7k-2)t+4$ and Lemma \ref{edge}, we derive $H\cong H^*$, and thus $\rho(G)\leq \rho(H)=\rho(H^*)$ since $G$ is a spanning subgraph of $H$ and Lemma \ref{rq1}. Therefore, $G=H\cong H^*$ by $\rho(G)\geq \rho(H^*)$.
	
	(ii) By Lemma \ref{q1} and (\ref{eq0}), we obtain
	\begin{equation}\label{eq2}
		2(n-(k-1)t-2)<q(H^*)\leq q(G)\leq \frac{2e(G)}{n-1}+n-2.
	\end{equation}
	By (\ref{eq2}) and direct computation, we deduce that $$
	e(G)>\frac{(n-1)(n-2(k-1)t-2)}{2}\geq\binom{n-(k-1)t-2}{2}+(k-1)t^2+(k+1)t+3,
	$$ where $n\geq \frac 12(k^2-1)t^2+\frac 12(5k-1)t+5$. Similar to the proof of (i), we have $H\cong H^*$. Then by Lemma \ref{rq1} and $q(G)\geq q(H^*)$, we have $q(G)\leq q(H)=q(H^*)$, and thus $G=H\cong H^*$ since $G$ is a spanning subgraph of $H$. 
\end{proof}
\vspace{0.5cm}

It is easy to see that (i) of Theorem \ref{ta} is an improvement of Theorem \ref{t1} since $(k-1)t^2+\frac12(3k+1)t+\frac92-(\frac 12(k-1)t^2+2kt-\frac t2+5)=\frac t2(k-1)(t-1)+\frac12(t-1)\geq0$. Besides, Theorem \ref{ta} generalizes Theorem \ref{t0} from $t=1$ to general $t$.

\begin{theorem}\label{tc}
	Let $t\geq 1$, $k\geq2$ and $n\geq(7k-2)t+4$ be integers, $H^*=K_t\vee(K_{n-kt-1}\cup(kt-t+1)K_1)$, $G$ be a $t$-connected graph of order $n$. If $\rho(\overline{G})\leq \sqrt{f(n,k,t)}$, then $G$ has a spanning $k$-tree unless $C_{n-(k-2)t-1}(G)\cong H^*$, where $f(n,k,t)=(1+\frac{(k-2)t}{n})((n-1-t)(kt-t+1)+n-\frac12(k-1)^2t^2-\frac52kt+\frac32t-5)$.
\end{theorem}

\begin{proof}
	Let $H=C_{n-(k-2)t-1}(G)$. Suppose to the contrary that $G$ has no spanning $k$-trees. Now we show $H\cong H^*$, where $H^*$ has no spanning $k$-trees by Proposition \ref{k-tree}. 
	
	Since $G$ has no spanning $k$-trees and Theorem \ref{closure}, $H$ has no spanning $k$-trees. Clearly, $H\ncong K_n$. By the definition of $C_{n-(k-2)t-1}(G)$, we have $d_H(u)+d_H(v)\leq n-(k-2)t-2$ for every pair of nonadjacent vertices $u$ and $v$ of $H$. Then for any $uv\in E(\overline{H})$, we obtain $$d_{\overline{H}}(u)+d_{\overline{H}}(v)= 2(n-1)-d_H(u)-d_H(v)\geq2n-2-(n-(k-2)t-2)=n+(k-2)t.$$ 
	
	It is easy to check that $$\sum_{v\in V(\overline{H})}d_{\overline{H}}(v)^2=\sum_{uv\in E(\overline{H})}(d_{\overline{H}}(u)+d_{\overline{H}}(v))\geq (n+(k-2)t)e(\overline{H}).$$
	Then we have $n\rho(\overline{H})^2\geq \sum\limits_{v\in V(\overline{H})}d_{\overline{H}}(v)^2 \geq (n+(k-2)t)e(\overline{H})$ by Lemma \ref{r2}.
	
	Let $f(n,k,t)=(1+\frac{(k-2)t}{n})((n-1-t)(kt-t+1)+n-\frac12(k-1)^2t^2-\frac52kt+\frac32t-5)$. Since $\overline{H}$ is a spanning subgraph of $\overline{G}$, we have $\rho(\overline{H})\leq\rho(\overline{G})\leq \sqrt{f(n,k,t)}$. Then 
	$$\begin{aligned}
		e(\overline{H})\leq&\frac{n\rho(\overline{H})^2}{n+(k-2)t}\leq \frac{nf(n,k,t)}{n+(k-2)t}\\=&(n-1-t)(kt-t+1)+n-\frac12(k-1)^2t^2-\frac52kt+\frac32t-5.
	\end{aligned}$$ 
	Thus, $$\begin{aligned}e(H)=&\binom{n}{2}-e(\overline{H})\\\geq&\binom{n}{2}-\left[(n-1-t)(kt-t+1)+n-\frac12(k-1)^2t^2-\frac52kt+\frac32t-5\right]\\=&\binom{n-(k-1)t-2}{2}+(k-1)t^2+(k+1)t+3.\end{aligned}$$ 
	
	By $n\geq(7k-2)t+4$ and Lemma \ref{edge}, we have $H\cong H^*$.
\end{proof}

Let $k=2$ and $t=1$ in Theorem \ref{tc}. Then we have the following corollary, which is an improvement of Theorem \ref{t3} for $n\geq 16$.

\begin{corollary}\label{coro10}
	Let $G$ be a connected graph of order $n\geq 16$. If $\rho(\overline{G})\leq \sqrt{3n-13}$, then $G$ has a Hamilton path unless $C_{n-1}(G)\cong K_1\vee (K_{n-3}\cup 2K_1)$.
\end{corollary}

Let $\mathcal{G}_n\text{ be the set of all graphs of order } n\geq 0 \text{, }\mathbb{G}_1(k,t)=\{G_t\vee (kt-t+2)K_1\mid G_t\in \mathcal{G}_t\}$$\text{ and }\mathbb{G}_2(k,t)=\{G_r\vee R(kt+2-r,t-r)\mid 0\leq r\leq t-3\text{ and } G_r \in\mathcal{G}_r\}.$ %where $G_r\vee R(kt+2-r,t-r)\cong R(kt+2,t)$ if $r=0$.
 
 Let $G\in\mathbb{G}_1(k,t)$. Then $G$ is a spanning subgraph of $K_t\vee(kt-t+2)K_1$. By Proposition \ref{k-tree}, we know that $K_t\vee(kt-t+2)K_1$ has no spanning $k$-trees, and thus $G$ has no spanning $k$-trees.

\begin{theorem}\label{td}
	Let $t\geq 1$, $k\geq2$ and $n\geq (k-2)t+4$ be integers, $G$ be a $t$-connected graph of order $n$ with $\rho(\overline{G})\leq \sqrt{(n-1-t)(kt-t+1)}$. 
	
	\noindent{\rm (i)} If $t\in \{1,2\}$, then $G$ has a spanning $k$-tree unless $G\in\mathbb{G}_1(k,t)$.
	
	\noindent{\rm (ii)} If $3\leq t\leq n-1$ and  $G\notin \mathbb{G}_2(k,t)$, then $G$ has a spanning $k$-tree unless $G\in \mathbb{G}_1(k,t)$.
\end{theorem}

\begin{proof}
	Let $H=C_{n-(k-2)t-1}(G)$. Suppose to the contrary in (i) and (ii) that $G$ contains no spanning $k$-trees. Then $H$ has no spanning $k$-trees by Theorem \ref{closure}. It is clear that $n\geq t+2$ since $H$ is $t$-connected and $H\ncong K_n$. 
	
	By the definition of $C_{n-(k-2)t-1}(G)$, we have $d_H(u)+d_H(v)\leq n-(k-2)t-2$ for any pair of nonadjacent vertices $u$ and $v$ of $H$. Then $$d_{\overline{H}}(u)+d_{\overline{H}}(v)= 2(n-1)-d_H(u)-d_H(v)\geq n+(k-2)t$$ for any $uv\in E(\overline{H})$, which implies any non-trivial component of $\overline{H}$ has a vertex with the degree at least $\frac{n+(k-2)t}{2}$. Suppose that $\overline{H}$ has $s(\geq 1)$ components. Then it is clear that $s(\frac{n+(k-2)t}{2}+1)>n$ for $s\geq 2$. Thus $\overline{H}$ has exactly one non-trivial component, which we denote by $F$. Besides, $d_{\overline{H}}(w)=0$ for any $w\in V(\overline{H})\backslash V(F)$. Now $d_F(u)+d_F(v)=d_{\overline{H}}(u)+d_{\overline{H}}(v)\geq n+(k-2)t$ for each $uv\in E(F)$. 
	
	Since $H$ is $t$-connected, we have $\delta(H)\geq t$. Then $d_H(u)\geq \delta(H)\geq t$ for each $u\in V(H)$, and thus $d_F(u)=d_{\overline{H}}(u)\leq n-1-t$ for each $u\in V(F)$. Now for each $uv\in E(F)$, we have $d_F(u)\geq n+(k-2)t-d_F(v)\geq n+(k-2)t-(n-1-t)=(k-1)t+1$. Similarly, we have $d_F(v)\geq (k-1)t+1$. Thus, $(k-1)t+1\leq d_F(u),d_F(v)\leq n-1-t$.
	
	 Let $f(d_F(u))=d_F(u)(n+(k-2)t-d_F(u))$. Then $d_F(u)d_F(v)\geq f(d_F(u)).$ Since $f(d_F(u))$ is a convex function on $d_F(u)\in[(k-1)t+1,n-1-t]$, we have $$\begin{aligned}
	 &d_{\overline{H}}(u)d_{\overline{H}}(v)=d_F(u)d_F(v)\geq f(d_F(u))\\\geq&\min\{f((k-1)t+1),f(n-1-t)\}=(n-1-t)(kt-t+1)\end{aligned}
	 $$ for any $uv\in E(\overline{H})$. Note that $G$ is a spanning subgraph of $H$, we have $\overline{H}$ is a spanning subgraph of $\overline{G}$. Then $\rho(\overline{G})\geq \rho(\overline{H})$ by Lemma \ref{rq1}. 
	 
	 By Lemma \ref{r3}, we obtain $$\begin{aligned}&\sqrt{(n-t-1)((k-1)t+1)}\geq \rho(\overline{G})\geq \rho(\overline{H})\\\geq& \min_{uv\in E(\overline{H})}\sqrt{d_{\overline{H}}(u)d_{\overline{H}}(v)}\geq \sqrt{(n-1-t)(kt-t+1)}.\end{aligned}$$ 
	Thus, we have $\rho(\overline{G})=\rho(\overline{H})=\min\limits_{uv\in E(\overline{H})}\sqrt{d_{\overline{H}}(u)d_{\overline{H}}(v)}=\sqrt{(n-1-t)(kt-t+1)}$. This implies that there exists an edge $uv\in E(\overline{H})$ such that $f(d_F(u))=d_F(u)d_F(v)=\min\limits_{uv\in E(\overline{H})}\sqrt{d_{\overline{H}}(u)d_{\overline{H}}(v)}$, where $d_{\overline{H}}(u)=(k-1)t+1$ and $d_{\overline{H}}(v)=n-1-t$. Since $F$ is the unique non-trivial component of $\overline{H}$, we have $$\rho(F)=\rho(\overline{H})=\min_{uv\in E(\overline{H})}\sqrt{d_{\overline{H}}(u)d_{\overline{H}}(v)}=\min_{uv\in E(\overline{H})}\sqrt{d_{F}(u)d_{F}(v)}.$$
	
	Since $F$ is connected and  Lemma \ref{r3}, we have $F$ is regular or semi-regular. 
	
	If $F$ is semi-regular, then we can assume that $d_F(u)=kt-t+1$ and $d_F(v)=n-1-t$ for any $uv\in E(F)$ by symmetry. By $n\geq |V(F)|\geq kt-t+1+n-1-t=n+(k-2)t\geq n$, we have $k=2$, and thus $\overline{H}=F\cong K_{t+1,n-t-1}$. However, $H=\overline{F}\cong K_{t+1}\cup K_{n-t-1}$ is disconnected, which contradicts the connectedness of $H$.
	
	If $F$ is regular, then for any $w\in V(F)$, we have $d_F(w)=d_F(v)=d_F(u)=kt-t+1=n-1-t$, which implies $n=kt+2$. Then $F\cong R(|V(F)|,kt-t+1)$.
	
	If $\overline{H}=F$, then $H=\overline{F}\cong \overline{R(kt+2,kt-t+1)}=R(kt+2,t)$. By $\rho(\overline{G})=\rho(\overline{H})$ and Lemma \ref{rq1}, we have $\overline{G}=\overline{H}$ and $G=H=R(kt+2,t)$ since $\overline{H}$ is connected.
	
	If $\overline{H}=F\cup rK_1(r\geq 1)$, then $\overline{H}\cong R(kt+2-r,kt-t+1)\cup rK_1$ and $1\leq r=kt+2-|V(F)|\leq kt+2-(kt-t+2)=t$, and thus $H\cong R(kt+2-r,t-r)\vee K_r$. Since $\rho(\overline{G})=\rho(\overline{H})$ and $\overline{H}$ is a spanning subgraph of $\overline{G}$, we have $\overline{G}\cong R(kt+2-r,kt-t+1)\cup \overline{G_r}$ and $G\cong R(kt+2-r,t-r)\vee G_r$, where $1\leq r\leq t$ and $G_r\in \mathcal{G}_r$. 
	
	Therefore, $G\cong R(kt+2-r,t-r)\vee G_r$ and $H\cong R(kt+2-r,t-r)\vee K_r$, where  $G_r\in \mathcal{G}_r$ and $0\leq r\leq t$. 
	
	{\noindent\textbf{Case 1.}} $t=1$.
	
	Now $0\leq r\leq 1$. If $r=0$, then $G\cong R(k+2,1)$ is disconnected, a contradiction. If $r=1$, then $G\cong R(k+1,0)\vee K_1\cong K_{1,k+1}\in \mathbb{G}_1(k,1)$.
	
	{\noindent\textbf{Case 2.}} $t=2$.
	
	Now $0\leq r\leq 2$. If $r=0$, then $G\cong R(2k+2,2)$. Since $G$ is connected, $G\cong R(2k+2,2)\cong C_{2k+2}$, where $C_{2k+2}$ is a cycle of order $2k+2$. However, $C_{2k+2}$ has a spanning $k$-tree, a contradiction. If $r=1$, then the sum of vertex degrees in $R(2k+2-r,2-r)$ is odd, a contradiction. If $r=2$, then $G\in \mathbb{G}_1(k,2)$.
	
	{\noindent\textbf{Case 3.}} $3\leq t\leq n-1$.
	
	Now we have $t-2\leq r\leq t$ by the assumption of (ii). 
	
	If $r=t-2$, then $H\cong R(kt-t+4,2)\vee K_{t-2}$. Since $H$ is $t$-connected, $R(kt-t+4,2)$ is connected, that is, $R(kt-t+4,2)\cong C_{kt-t+4}$. Then $H\cong C_{kt-t+4}\vee K_{t-2}$ has a Hamilton path. However, $H$ has a spanning $k$-tree, a contradiction.
	
	If $r=t-1$, then $H\cong R(kt-t+3,1)\vee K_{t-1}$. Note that $R(kt-t+3,1)$ is disconnected, which contradicts that $H$ is $t$-connected. 
	
	If $r=t$, then $G\in \mathbb{G}_1(k,t)$. 
\end{proof}

\vspace{0.5cm}

Theorem \ref{tc} and Theorem \ref{td} are relatively similar in their conclusions, both considering the upper bound of the adjacent spectral radius of the complement graph. However, when $n\geq \frac 12(k-1)^2t^2+\frac52kt-\frac32 t+5$, Theorem \ref{tc} is more advantageous, while for smaller values of \( n \), the results of Theorem \ref{td} are better.

Let $\mathbb{G}_3(n,k,t)=\{G_r\vee R(n-r,\frac{n-(k-2)t-2}{2}-r)\mid 0\leq r\leq \frac{n-(k-2)t-2}{2},G_r \in \mathcal{G }_r\}$, where $\frac{n-(k-2)t-2}{2}$ is an integer. Note that $G_r\vee R(n-r,\frac{n-(k-2)t-2}{2}-r)=R(n,\frac{n-(k-2)t-2}{2})$ if $r=0$.
\begin{theorem}\label{te}
	Let $t\geq 1$ and $k\geq2$ be integers, $G$ be a $t$-connected graph of order $n\geq (k-2)t+4$. If $q(\overline{G})\leq n+(k-2)t$ and $G\notin \mathbb{G}_3(n,k,t)$, then $G$ has a spanning $k$-tree.
\end{theorem}
\begin{proof}
	Let $H=C_{n-(k-2)t-1}(G)$. Then $H$ is $t$-connected. Suppose to the contrary that $G$ has no spanning $k$-trees. Then $H$ has no spanning $k$-trees by Theorem \ref{closure}, and thus $H\ncong K_n$. By the definition of $C_{n-(k-2)t-1}(G)$, we have $d_H(u)+d_H(v)\leq n-(k-2)t-2$ for every pair of nonadjacent vertices $u$ and $v$ of $H$. Then $$d_{\overline{H}}(u)+d_{\overline{H}}(v)= n-1-d_H(u)+n-1-d_H(v)\geq2n-2-(n-(k-2)t-2)=n+(k-2)t$$ for any $uv\in E(\overline{H})$. By $$z(\overline{H})=\sum_{v\in V(\overline{H})}d_{\overline{H}}(v)^2=\sum_{uv\in E(\overline{H})}(d_{\overline{H}}(u)+d_{\overline{H}}(v))\geq \left(n+(k-2)t\right)e(\overline{H})$$
	and Lemma \ref{q2}, we have $q(\overline{H})\geq \frac{z(\overline{H})}{e(\overline{H})}\geq n+(k-2)t$. Since $\overline{H}$ is a spanning subgraph of $\overline{G}$ and Lemma \ref{rq1}, we have $q(\overline{G})\geq q(\overline {H})$. Combining with $q(\overline{G})\leq n+(k-2)t$, it is clear that $$n+(k-2)t\geq q(\overline{G})\geq q(\overline{H})\geq n+(k-2)t.$$ Then $q(\overline{G})=q(\overline{H})= \frac{z(\overline{H})}{e(\overline{H})}= n+(k-2)t$, and thus $d_{\overline{H}}(u)+d_{\overline{H}}(v)=n+(k-2)t$ for any $uv\in E(\overline{H})$. Through a discussion similar to that in Theorem \ref{td}, we know that $\overline{H}$ contains exactly one non-trivial component, which we denote by $F$. 
	
	Clearly, $q(F)=q(\overline{H})= \frac{z(\overline{H})}{e(\overline{H})}=\frac{z(F)}{e(F)}$. Then $F$ is regular or semi-regular with $|V(F)|\geq \frac{n+(k-2)t}{2}+1$ by Lemma \ref{q2} and the connectedness of $F$. 
	
	Suppose that $F$ is semi-regular. Then $F=\overline{H}$ is a complete bipartite graph, and thus $H=\overline{F}$ is disconnected, a contradiction. Hence $F$ is $\frac{n+(k-2)t}{2}$-regular by $d_{\overline{H}}(u)+d_{\overline{H}}(v)=n+(k-2)t$ for any $uv\in E(\overline{H})$. 
	
	If $\overline{H}=F$, then $\overline{H}$ is connected, and thus $\overline{G}=\overline{H}$  by $q(\overline{G})=q(\overline{H})$ and Lemma \ref{rq1}. However, it is obvious that $G=H=\overline{F}\cong R(n,\frac{n-(k-2)t-2}{2})\in \mathbb{G}_3(n,k,t)$, a contradiction. 
	
	If $\overline{H}=F\cup rK_1$, then $r=n-|V(F)|$ and $1\leq r\leq \frac{n-(k-2)t-2}{2}$. Clearly, we have $\overline{H}=R(n-r,\frac{n+(k-2)t}{2})\cup rK_1$. Then $\overline{G}=R(n-r,\frac{n+(k-2)t}{2})\cup G_r$, where $G_r\in \mathcal{G}_r$ by $q(\overline{H})=q(\overline{G})$ and Lemma \ref{rq1}. It is easy to check that $G\in \mathbb{G}_3(n,k,t)$, a contradiction.
\end{proof}

\vspace{0.5cm}

Theorem \ref{te} extends Theorem \ref{t4} from $k=2$ to general $k$. 

Let $ \mathbb{G}_4(n)=\{G_r\vee R(n-r,\frac{n-2}{2}-r)\mid 0\leq r\leq \frac{n-8}{2},G_r\in \mathcal{G}_r\}$ and $\mathbb{G}_5(n,r)=\{G_r\vee R(n-r,\frac{n-2}{2}-r)\mid G_r\in \mathcal{G}_r\}$ for $\frac{n-6}{2}\leq r\leq \frac{n-2}{2}$. In Theorems \ref{td} and \ref{te}, we obtain Corollary \ref{coro13} by taking $k=2$.
\begin{corollary}\label{coro13}
	Let $t\geq 1$ be an integer, $G$ be a $t$-connected graph of order $n\geq 4$. 
	
	\noindent{\rm (i)} For $t\in \{1,2\}$, if $\rho(\overline{G})\leq \sqrt{(n-1-t)(t+1)}$, then $G$ has a Hamilton path unless $G\in \mathbb{G}_1(2,t)$.
	
	\noindent{\rm (ii)} For $3\leq t\leq n-1$, if $\rho(\overline{G})\leq \sqrt{(n-1-t)(t+1)}$ and  $G\notin\mathbb{G}_2(2,t)$, then $G$ has a Hamilton path unless $G\in \mathbb{G}_1(2,t)$. 
	
	\noindent{\rm (iii)} If $n\geq 9$, $q(\overline{G})\leq n$ and $G\notin\mathbb{G}_4(n)$, then $G$ has a Hamilton path unless $G\in \mathbb{G}_5(n,\frac{n-2}{2})$.
\end{corollary}

\begin{proof}
	By Theorem \ref{td}, we obtain (i) and (ii). Now we prove (iii) by Theorem \ref{te}. Clearly, $\mathbb{G}_3(n,2,t)=\mathbb{G}_4(n)\cup \mathbb{G}_5(n,\frac{n-2}{2})\cup \mathbb{G}_5(n,\frac{n-4}{2})\cup \mathbb{G}_5(n,\frac{n-6}{2})$. If $n$ is odd, then $G\notin \mathbb{G}_3(n,2,t)$, and thus $G$ has a Hamilton path by the definition of $\mathbb{G}_3(n,2,t)$ and the case $k=2$ in Theorem \ref{te}. Next we consider $n$ is even.
	
	If $G\notin \mathbb{G}_3(n,2,t)$, then $G$ has a Hamilton path by Theorem \ref{te}. 
	
	If $G\in \mathbb{G}_5(n,\frac{n-2}{2})$, then $ G=G_\frac{n-2}{2}\vee\frac{n+2}{2}K_1$($G_\frac{n-2}{2}\in \mathcal{G}_\frac{n-2}{2})$, and thus $C_{n-1}(G)=K_\frac{n-2}{2}\vee\frac{n+2}{2}K_1$. Since $C_{n-1}(G)$ has no Hamilton paths, $G$ has no Hamilton paths by the case $k=2$ in Theorem \ref{closure}. 
	
	If $G\in \mathbb{G}_5(n,\frac{n-6}{2})$, then $C_{n-1}(G)=K_{\frac{n-6}{2}}\vee R(\frac{n+6}{2},2)$. Since the number of components of $R\left(\frac{n+6}{2},2\right)$ is at most $\left\lfloor \frac{n+6}{6} \right\rfloor$ and does not exceed \( \frac{n-6}{2}+1 \), it follows that $C_{n-1}(G)$ has a Hamilton path. Then $G$ has a Hamilton path by the case $k=2$ in Theorem \ref{closure}.
	
	Similarly, if $G\in \mathbb{G}_5(n,\frac{n-4}{2})$, then $G$ has a Hamilton path.
\end{proof}

\begin{remark}
	%In this section, we obtain many conclusions, which improve or generalize some well-known results. Among these conclusions, (i) of Theorem \ref{ta} is an improvement of Theorem \ref{t1}; Theorem \ref{ta} generalizes Theorem \ref{t0} from $t=1$ to general $t$; In Theorem \ref{tc}, we obtain Corollary \ref{coro10} by letting $k=2$ and $t=1$, which is an improvement of Theorem \ref{t3} when $n\geq 16$; Theorem \ref{te} extends Theorem \ref{t4} from $k=2$ to general $k$; In Theorems \ref{td} and \ref{te}, we obtain Corollary \ref{coro13} by taking $k=2$, which is a spectral result related to the Hamilton path.	
	In this section, we obtain several conclusions that improve or generalize some well-known results. Specifically:
	\begin{itemize}
		\item Theorem \ref{ta} (i) improves Theorem \ref{t1}.
		\item Theorem \ref{ta} generalizes Theorem \ref{t0} from \( t = 1 \) to general \( t \).
		\item In Theorem \ref{tc}, by setting \( k = 2 \) and \( t = 1 \), we obtain Corollary \ref{coro10}, which improves Theorem \ref{t3} for \( n \geq 16 \).
		\item Theorem \ref{te} extends Theorem \ref{t4} from \( k = 2 \) to general \( k \).
		\item In Theorems \ref{td} and \ref{te}, by taking \( k = 2 \), we derive Corollary \ref{coro13}, which is a spectral result related to the Hamilton path.
	\end{itemize}
\end{remark}

\section{Spanning $k$-ended-trees}\label{sec4}
\hspace{1.5em} In this section, we present some spectral conditions for the existence of spanning $k$-ended-trees in $t$-connected graphs, which improve the results of Ao et al. \cite{GYA1}. Now we recall some relevant conclusions. 

In \cite{HB}, Broersma and Tuinstra presented the following closure theorem for the existence of a spanning $k$-ended-tree in a connected graph.

\begin{theorem}{\rm(\!\!\cite{HB})}\label{closure2}
	Let $G$ be a connected graph of order $n$, $k$ be an integer with $2 \leq k \leq n -1$. Then $G$ has 
	a spanning $k$-ended-tree if and only if the $(n -1)$-closure $C_{n-1}(G)$ of $G$ has a spanning $k$-ended-tree.
\end{theorem}

In \cite{GYA1}, Ao, Liu and Yuan proved the following spectral conditions to ensure the existence of a spanning $k$-ended-tree in connected graphs by Theorem \ref{closure2}.

\begin{theorem}{\rm(\!\!\cite{GYA1})}\label{t2}
	Let $k\geq 2$ be an integer, $G$ be a connected graph of order $n$.
	
	\noindent{\rm (i)} If $n \geq \max\{6k + 5, \frac3
	2k^2 + \frac32k + 2\}$ and $q(G) \geq q(K_1 \vee(K_{n-k-1} \cup kK_1))$, then $G$ has a spanning $k$-ended-tree unless $G\cong K_1 \vee(K_{n-k-1} \cup kK_1)$.
	
	\noindent{\rm (ii)}	If $\rho(\overline{G})\leq\sqrt{k(n-2)}$, then $G$ has a spanning $k$-ended-tree unless $G\cong K_{1,k+1}$.
\end{theorem}

Motivated by \cite{HB,GYA1}, we study some spectral conditions to ensure the existence of a spanning $k$-ended-tree in a $t$-connected graph. 

\begin{theorem}\label{tf}
	Let $k\geq 2$, $t\geq 1$ and $n\geq \max\{6k+6t-1,\frac 32(k+t-1)^2+\frac32k+\frac32t+\frac12\}$ be integers, $G^*=K_t\vee(K_{n-k-2t+1}\cup (k+t-1)K_1)$, $G$ be a $t$-connected graph of order $n$. If $q(G)\geq q(G^*)$, then $G$ has a spanning $k$-ended-tree unless $G\cong G^*$.
\end{theorem}

\begin{proof}
Suppose to the contrary that $G$ has no spanning $k$-ended-trees. Now we show $G\cong G^*$, where $ G^*$ has no spanning $k$-ended-trees by Lemma \ref{k-ended tree}. 
	
	It is clear that we have $q(G^*)>q(K_{n-k-t+1}\cup (k+t-1)K_1))=2(n-k-t)$ by Lemma \ref{rq1}. By Lemma \ref{q1}, we obtain
	$$\begin{aligned}2(n-k-t)<q(G^*)\leq q(G)\leq \frac{2e(G)}{n-1}+n-2.\end{aligned}$$
	
	Then we have $
	e(G)>\frac{(n-1)(n-2k-2t+2)}{2}\geq\binom{n-k-t}{2}+(k+t-1)^2+k+t
	$ for $n\geq \frac 32(k+t-1)^2+\frac32k+\frac32t+\frac12$. Let $H=C_{n-1}(G)$. Then $H$ has no spanning $k$-ended-trees by Theorem \ref{closure2}. By $n\geq \max\{6k+6t-1,\frac 32(k+t-1)^2+\frac32k+\frac32t+\frac12\}\geq \max\{6k+6t-1,k^2+kt+t+1\}$ and Lemma \ref{edge2}, we have $H\cong G^*$. Thus $q(G)\leq q(H)=q(G^*)$ since $G$ is a spanning subgraph of $H$ and Lemma \ref{rq1}. Combining $q(G)\geq q(G^*)$, we have $G= H\cong G^*$. Then we complete the proof.
\end{proof}

\vspace{0.5cm}

In Theorem \ref{tf}, we can obtain (i) of Theorem \ref{t2} by setting $t=1$.

\begin{theorem}\label{tg}
	Let $k\geq 2$, $t\geq 1$ and $n\geq\max\{6k+6t-1,k^2+kt+t+1\}$ be integers, $G^*=K_t\vee(K_{n-k-2t+1}\cup (k+t-1)K_1)$, $G$ be a $t$-connected graph of order $n$. If $\rho(\overline{G})\leq \sqrt{(k+t)(n-2t)-\frac32k^2+\frac k2-kt+\frac12t^2+\frac12t-1}$, then $G$ has a spanning $k$-ended-tree unless $C_{n-1}(G)\cong G^*$.
\end{theorem}

\begin{proof}
	Let $H=C_{n-1}(G)$. Suppose to the contrary that $G$ has no spanning $k$-ended-trees, then $H$ has no spanning $k$-ended-trees by Theorem \ref{closure2}, and thus $H\ncong K_n$. Now we show $H\cong G^*$, where $G^*$ has no spanning $k$-ended-trees by Lemma \ref{k-ended tree}. 
	
	By the definition of $C_{n-1}(G)$, we have $d_H(u)+d_H(v)\leq n-2$ for every pair of nonadjacent vertices $u$ and $v$ of $H$. Then we have $$d_{\overline{H}}(u)+d_{\overline{H}}(v)= 2(n-1)-d_H(u)-d_H(v)\geq2(n-1)-(n-2)=n,$$ for any $uv\in E(\overline{H})$. Summing these inequalities for all $uv\in E(\overline{H})$, we obtain $$\sum_{v\in V(\overline{H})}d_{\overline{H}}(v)^2=\sum_{uv\in E(\overline{H})}(d_{\overline{H}}(u)+d_{\overline{H}}(v))\geq ne(\overline{H}).$$
	By Lemma \ref{r2}, we have $n\rho(\overline{H})^2\geq \sum\limits_{v\in V(\overline{H})}d_{\overline{H}}(v)^2 \geq ne(\overline{H}).$
	
	Since $\overline{H}$ is a spanning subgraph of $\overline{G}$ and Lemma \ref{rq1}, we have $$\rho(\overline{H})\leq\rho(\overline{G})\leq \sqrt{(k+t)(n-2t)-\frac32k^2+\frac k2-kt+\frac12t^2+\frac12t-1}.$$ Thus 
	$e(\overline{H})\leq\rho(\overline{H})^2\leq (k+t)(n-2t)-\frac32k^2+\frac k2-kt+\frac12t^2+\frac12t-1.$ 
	Therefore, we obtain $e(H)=\binom{n}{2}-e(\overline{H})\geq\binom{n-k-t}{2}+(k+t-1)^2+k+t$, and thus $H\cong G^*$ by Lemma \ref{edge2}.
\end{proof}

\vspace{0.5cm}

In Theorem \ref{tg}, by taking $t=1$, we obtain Corollary \ref{coro4.5}, which improves (ii) of Theorem \ref{t2} when $n\geq \max\{6k+5,\frac32k^2+\frac12k+2\}$.	

\begin{corollary}\label{coro4.5}
	Let $k\geq 2$ and $n\geq\max\{6k+5,k^2+k+2\}$ be integers, $G$ be a connected graph of order $n$. If $\rho(\overline{G})\leq \sqrt{k(n-2)+n-2-\frac32k^2-\frac12k}$, then $G$ has a spanning $k$-ended-tree unless $C_{n-1}(G)\cong K_1\vee(K_{n-k-1}\cup kK_1)$.
\end{corollary}

\section*{Funding}
	\hspace{1.5em}This work is supported by the National Natural Science Foundation of China (Grant Nos. 12371347).

\end{document}